\documentclass[11pt]{amsart}
\usepackage{amsmath, amsthm, thmtools, amssymb, colonequals, mathrsfs, environ,
tabu, stmaryrd, tikz, enumitem, thm-restate, fullpage}
\usepackage[T1]{fontenc}

\usepackage[
	backref,
	pdfauthor={Vishal Arul}, 
  pdftitle={On the ell-adic valuation of certain Jacobi sums},
]{hyperref}
\usepackage[alphabetic,backrefs,lite,nobysame]{amsrefs} 
\usepackage{nameref, cleveref}

\hypersetup{
    colorlinks,
    linkcolor={red!50!black},
    citecolor={blue!50!black},
    urlcolor={blue!80!black}
}

\declaretheorem[name=Lemma,
style=plain,
sharenumber=theorem,
refname={Lemma,Lemmas},
Refname={Lemma,Lemmas}]{lemma}

\declaretheorem[name=Corollary,
style=plain,
sharenumber=theorem,
refname={Corollary,Corollaries},
Refname={Corollary,Corollaries}]{corollary}

\declaretheorem[name=Definition,
style=definition,
sharenumber=theorem,
refname={Definition,Definitions},
Refname={Definition,Definitions}]{definition}

\makeatletter
\renewcommand{\itemautorefname}{\@gobble}
\makeatother

\crefname{enumi}{}{}

\makeatletter
\renewcommand*{\eqref}[1]{%
  \hyperref[{#1}]{\textup{\tagform@{\ref*{#1}}}}%
}
\makeatother

\numberwithin{equation}{section}

\SetSymbolFont{stmry}{bold}{U}{stmry}{m}{n}

\DeclareMathOperator{\ind}{ind}

\makeatletter
\g@addto@macro\bfseries{\boldmath} 
\makeatother

\begin{document}
\author{Vishal Arul}
\title{On the $\ell$-adic valuation of certain Jacobi sums}
\thanks{This research was supported in part by grants from the Simons Foundation
(\#402472 to Bjorn Poonen, and \#550033).}

\begin{abstract}
Jacobi sums are ubiquitous in number theory, and congruences often provide a
helpful way to study them. A $p$-adic congruence for Jacobi sums comes from
Stickelberger's congruence, and various $\ell$-adic congruences have been
studied in
\cites{evans1998congruences,ihara1986profinite,iwasawa1975note,miki1987adic,uehara1987congruence}.
We establish a new $\ell$-adic congruence for certain Jacobi sums.
\end{abstract}

\maketitle

\section{Introduction}

The Jacobi sum is usually defined as follows.
\begin{definition}
Fix a finite field $\mathbf{F}_{q}$, a field $L$, and two nontrivial
multiplicative characters $\chi, \psi : \mathbf{F}_{q}^{\times} \to L^{\times}$.
Then the Jacobi sum $J(\chi, \psi)$ is 
\[
J(\chi, \psi) \colonequals \sum_{x \in \mathbf{F}_{q} \setminus \{ 0, 1 \}}
\chi(x) \psi(1 - x) \in L.
\]
\end{definition}

Jacobi sums have various applications in number theory; see
\cite{berndt1998gauss} for many examples. They appear as Frobenius eigenvalues
for the Fermat curve in \cite{katz1981crystalline} and quotients of the Fermat
curve in \cites{ArulTorsion,jkedrzejak2014torsion,jkedrzejak2016note}. In
particular, it is useful to have congruences for Jacobi sums
\cites{conrad1995jacobi,evans1998congruences,ihara1986profinite,iwasawa1975note,miki1987adic,uehara1987congruence}.
We prove a new congruence for Jacobi sums of the type considered by Uehara
\cite{uehara1987congruence}. The main result of this paper is used in
\cite{ArulTorsion} to study torsion points on the curve $y^{n} = x^{d} + 1$,
which is a quotient of the Fermat curve $F_{n d} : X^{n d} + Y^{n d} + Z^{n d} =
0$.

In \cite{uehara1987congruence}, Uehara establishes an $\ell$-adic congruence for
Jacobi sums of the form $J(\chi, \chi^{c f})$ where $\chi :
\mathbf{F}_{q}^{\times} \to \mathbf{C}^{\times}$ is a character of order $\ell
f$ and $q \equiv 1 \bmod{\ell f}$. Certain cyclotomic units of
$\mathbf{Q}(\zeta_{\ell f})$ appear in Uehara's expansion.  Our setup will be
very similar to that of Uehara's, and we will also find a connection with
cyclotomic units.

Fix two distinct primes $\ell$ and $f$, a finite field $\mathbf{F}_{q}$
satisfying $q \equiv 1 \pmod{\ell f}$, and a primitive $\ell f$th root of unity
$\zeta_{\ell f} \in \overline{\mathbf{Q}}$. Let
\begin{align*}
L &\colonequals  \mathbf{Q}(\zeta_{\ell f}) \\
\mathcal{O}_{L} &\colonequals \mathbf{Z}[\zeta_{\ell f}] \\
\zeta_{f} &\colonequals \zeta_{\ell f}^{\ell} \\
M &\colonequals \mathbf{Q}(\zeta_{f}) \\
\mathcal{O}_{M} &\colonequals \mathbf{Z}[\zeta_{f}] \\
\zeta_{\ell} &\colonequals \zeta_{\ell f}^{f} \\
\pi_{\ell} &\colonequals \zeta_{\ell} - 1.
\end{align*}
Let $\xi_{\ell}, \xi_{f} \in \overline{\mathbf{Q}}$ be $\ell$th and $f$th roots
of unity such that 
\begin{align*}
\xi_{\ell}^f = \zeta_{\ell} \\
\xi_f^\ell = \zeta_f. 
\end{align*}
Let $\chi : \mathbf{F}_{q}^{\times} \to L^\times$ be a character of order $\ell
f$. 

Let $g$ be a generator of the multiplicative group $\mathbf{F}_{q}^{\times}$ and
abuse notation to define $\zeta_{\ell f} \colonequals g^{(q - 1) / (\ell f)}$
and $\zeta_{f}, \zeta_{\ell}, \xi_{f}, \xi_{\ell}$ analogously to be elements of
$\mathbf{F}_{q}^\times$.

\begin{lemma}
\label{Lemma:ZetaXiConnection}
$\zeta_{\ell f} = \xi_{\ell} \xi_{f}$.
\end{lemma}
\begin{proof}
Since $\ell$ and $f$ are coprime, $\xi_{\ell}$ is the unique $\ell$th root of
unity such that $\xi_{\ell}^f = \zeta_{\ell}$. Since
\begin{align*}
\left(\frac{\zeta_{\ell f}}{\xi_{f}}\right)^{\ell} &=
\frac{\zeta_{f}}{\zeta_{f}} = 1; \text{ and } \\
\left(\frac{\zeta_{\ell f}}{\xi_{f}}\right)^{f} &= \frac{\zeta_{\ell}}{1} =
\zeta_{\ell},
\end{align*}
we are done.
\end{proof}

\begin{definition}
For integers $a, b$, define
\[
J(a, b) := J(\chi^{a}, \chi^{b}) = \sum_{x \in \mathbf{F}_{q} \setminus \{ 0, 1
\}} \chi^{a}(x) \chi^{b}(1 - x) \in \mathcal{O}_{L}.
\]
\end{definition}
\begin{definition}
\label{Definition:DefineEtaIJ}
For $i \in [0, \ell - 1]$ and $j \in [1, f - 1]$, define
\[
\eta_{i, j} \colonequals \prod_{r = 0}^{\ell - 1} \left( 1 - \xi_{\ell}^{r}
\xi_{f}^{j}  \right)^{\binom{r}{i}} \in \mathbf{F}_{q}^\times.
\]
\end{definition}

Our main result is the following.
\begin{restatable*}{theorem}{CongruenceUpToEllMinusOne}
\label{Theorem:CongruenceUpToEllMinus1}
For $k \in [1, \ell - 1]$, the following are equivalent:
\begin{enumerate}[label=\upshape(\arabic*),
ref=\autoref{Theorem:CongruenceUpToEllMinus1}(\arabic*)]

\item \label{Theorem:CongruenceUpToEllMinus1Congruence}
$J(\ell, f) + 1 \in \pi_{\ell}^k \mathcal{O}_{L}$;

\item \label{Theorem:CongruenceUpToEllMinus1Units}
$\eta_{i, j} \in \mathbf{F}_{q}^{\times \ell}$
for all $i \in [0, k - 2]$ and $j \in [1, f - 1]$; 

\item \label{Theorem:CongruenceUpToEllMinus1UnitsHalf}
$\eta_{i, j} \in \mathbf{F}_{q}^{\times \ell}$ for all $i \in [0, k - 2]$ and $j
\in [1, f / 2]$. 

\end{enumerate}
In particular, $J(\ell, f) + 1 \in \pi_{\ell} \mathcal{O}_{L}$ always holds.
\end{restatable*}

Our methods allow us to even reach the case $k = \ell$, which we analyze in
\autoref{Section:kEqualsEll}.

\begin{restatable*}{theorem}{CongruencekEqualsEll}
\label{Theorem:CongruencekEqualsEll}
The following are equivalent:

\begin{enumerate}[label=\upshape(\arabic*),
ref=\autoref{Theorem:CongruencekEqualsEll}(\arabic*)]

\item \label{Theorem:CongruencekEqualsEllCongruence}
$J(\ell, f) + 1 \in \pi_{\ell}^{\ell} \mathcal{O}_{L}$

\item \label{Theorem:CongruencekEqualsEllUnits}
$q \equiv 1 \pmod{\ell^{2} f}$ and $1 - \xi_{\ell}^{i} \xi_{f}^{j} \in
\mathbf{F}_{q}^{\times \ell}$ for all $i \in [0, \ell - 1]$ and $j \in [1, f -
1]$;

\item \label{Theorem:CongruencekEqualsEllUnitsHalf}
$q \equiv 1 \pmod{\ell^{2} f}$ and $1 - \xi_{\ell}^{i} \xi_{f}^{j} \in
\mathbf{F}_{q}^{\times \ell}$ for all $i \in [0, \ell - 1]$ and $j \in [1, f /
2]$.

\end{enumerate}
\end{restatable*}

\section{A few properties of binomial coefficients}
\begin{lemma}
\label{Lemma:BinomialCongruences}\hfill
\begin{enumerate}[label=\upshape(\arabic*),
ref=\autoref{Lemma:BinomialCongruences}(\arabic*)]

\item \label{Lemma:BinomialCongruencesSymmetry}
For $a \in \mathbf{Z}$ and $b \in [0, a]$,
\[
\binom{a}{b} = \binom{a}{a - b}.
\]

\item \label{Lemma:BinomialCongruencesDefinition}
For $a \in \mathbf{Z}$ and $b \in \mathbf{Z}_{\ge 0}$,
\[
\binom{a}{b + 1} = \frac{a}{b + 1} \binom{a - 1}{b}.
\]

\item \label{Lemma:BinomialCongruencesPascal}
For $a \in \mathbf{Z}$ and $b \in \mathbf{Z}_{\ge 0}$,
\[
\binom{a}{b + 1} = \binom{a - 1}{b} + \binom{a - 1}{b + 1}.
\]

\item \label{Lemma:BinomialCongruencesPascalModified}
For $a \in \mathbf{Z}$ and $b \in \mathbf{Z}_{\ge 0}$,
\[
a \binom{a}{b} = (b + 1) \binom{a}{b + 1} + b \binom{a}{b}.
\]

\item \label{Lemma:BinomialCongruencesHockeystick}
For $a \in \mathbf{Z}$ and $b \in \mathbf{Z}_{\ge 0}$,
\[
\sum_{c = 0}^{a - 1} \binom{c}{b} = \binom{a}{b + 1}.
\]

\item \label{Lemma:BinomialCongruencesVandermonde}
For $a, b \in \mathbf{Z}$ and $c \in \mathbf{Z}_{\ge 0}$,
\[
\binom{a + b}{c} = \sum_{d = 0}^{c} \binom{a}{d} \binom{b}{c - d}.
\]

\item \label{Lemma:BinomialCongruencesNegative}
For $a \in \mathbf{Z}$ and $b \in \mathbf{Z}_{\ge 0}$,
\[
\binom{-a}{b} = (-1)^{b} \sum_{c = 0}^{b} \binom{b - 1}{b - c} \binom{a}{c}.
\]

\item \label{Lemma:BinomialCongruencesPascalModifiedHockeystick}
For $a \in \mathbf{Z}$ and $b \in \mathbf{Z}_{\ge 0}$,
\[
\sum_{c = 0}^{a - 1} c \binom{c}{b} = (a - 1) \binom{a}{b + 1} - \binom{a}{b +
2}. 
\]

\item \label{Lemma:BinomialCongruencesReduceTop}
For $a_{1}, a_{2} \in \mathbf{Z}$ and $b \in [0, \ell - 1]$ such that $a_{1}
\equiv a_{2} \pmod{\ell}$, 
\[
\binom{a_1}{b} \equiv \binom{a_2}{b} \pmod{\ell}.
\]

\item \label{Lemma:BinomialCongruences0}
For $a \in \mathbf{Z}$ and $b \in \mathbf{Z}_{\ge 0}$ such that $a \equiv 0
\pmod{\ell}$, $b \not\equiv 0 \pmod{\ell}$, 
\[
\binom{a}{b} \equiv 0 \pmod{\ell}.
\]

\item \label{Lemma:BinomialCongruencesReduce}
For $a \in \mathbf{Z}$ and $b \in \mathbf{Z}_{\ge 0}$,
\[
\binom{a \ell}{b \ell} \equiv \binom{a}{b} \pmod{\ell}.
\]

\end{enumerate}
\end{lemma}
\begin{proof}
For \autoref{Lemma:BinomialCongruencesSymmetry}, use 
\[
\binom{n}{k} = \frac{n!}{k! (n - k)!}. 
\]
For \autoref{Lemma:BinomialCongruencesDefinition} --
\autoref{Lemma:BinomialCongruencesPascalModified}, use 
\[
\binom{x}{k} = \frac{x (x - 1) \cdots (x - (k - 1)) }{ k! }.
\]

\begin{enumerate}[label=\upshape(\arabic*),
ref={the proof of \autoref{Lemma:BinomialCongruences}(\arabic*)}]
\setcounter{enumi}{4}

\item \label{LemmaProof:BinomialCongruencesHockeystick}
Induct on $a$ and use \autoref{Lemma:BinomialCongruencesPascal}.

\item \label{LemmaProof:BinomialCongruencesVandermonde}
This is Vandermonde's identity for binomial coefficients, and it follows
by comparing the $x^{c}$-coefficient of both sides of $(1 + x)^{a + b} = (1 +
x)^{a} (1 + x)^{b}$.

\item \label{LemmaProof:BinomialCongruencesNegative}
Note that 
\[
\binom{-a}{b} = \frac{(-a)(-a - 1)\cdots(-a - (b - 1))}{b!} = (-1)^{b} \binom{a + b - 1}{b},
\]
so we are done by applying \autoref{Lemma:BinomialCongruencesVandermonde}.

\item \label{LemmaProof:BinomialCongruencesPascalModifiedHockeystick}
We have
\begin{align*}
\sum_{c = 0}^{a - 1} c \binom{c}{b} &= \sum_{c = 0}^{a - 1} \left( (b + 1)
\binom{c}{b + 1} + b \binom{c}{b} \right) &&\text{(by
\autoref{Lemma:BinomialCongruencesPascalModified})} \\
&= (b + 1) \binom{a}{b + 2} + b \binom{a}{b + 1} &&\text{(by
\autoref{Lemma:BinomialCongruencesHockeystick})} \\
&= a \binom{a}{b + 1} - \left(\binom{a}{b + 2} + \binom{a}{b + 1} \right)
&&\text{(by \autoref{Lemma:BinomialCongruencesPascalModified})} \\
&= (a - 1) \binom{a}{b + 1} - \binom{a}{b + 2}.
\end{align*}

\item \label{LemmaProof:BinomialCongruencesReduceTop}
Consider the polynomial $q(x) \colonequals \binom{x}{b} \in
\mathbf{F}_{\ell}[x]$. It follows from $b! p(x) = x (x - 1) \cdots (x - (b -
1))$ that $b! p(a_1) \equiv b!  p(a_2) \pmod{\ell}$. Since $b \in [0, \ell -
1]$, $b!$ is invertible modulo $\ell$, so we may divide both sides by $b!$ to
get $p(a_1) \equiv p(a_2) \pmod{\ell}$.

\item \label{LemmaProof:BinomialCongruences0}
For any $i$, note that $\binom{a}{i} \pmod{\ell}$ is the $x^i$-coefficient of
the polynomial $p(x) \colonequals (1 + x)^a \in \mathbf{F}_{\ell}[x]$. We have
$p(x) = ((1 + x)^{\ell})^{a / \ell} = (1 + x^{\ell})^{a / \ell}$, so since $b
\nmid \ell$, $\binom{a}{b} = [x^b] p(x) \equiv 0 \pmod{\ell}$.

\item \label{LemmaProof:BinomialCongruencesReduce}
As in the previous part, define $p(x) \colonequals (1 + x)^a \in
\mathbf{F}_{\ell}[x]$. Then 
\begin{equation}
\label{Equation:BinomabAsCoefficient}
\binom{a}{b} \equiv [x^{b \ell}] p(x^{\ell}) \pmod{\ell},
\end{equation}
and since $p(x^{\ell}) = (1 + x^{\ell})^{a} = (1 + x)^{a \ell}$,
\begin{equation}
\label{Equation:BinomaellbellAsCoefficient}
[x^{b \ell}] p(x^{\ell}) \equiv \binom{a \ell}{b \ell} \pmod{\ell},
\end{equation}
so we finish by combining \eqref{Equation:BinomabAsCoefficient} and
\eqref{Equation:BinomaellbellAsCoefficient}. \qedhere
\end{enumerate}
\end{proof}

\section{The index}

Recall that $g$ is a generator of the multiplicative group
$\mathbf{F}_{q}^{\times}$.
\begin{definition}
For $x \in \mathbf{F}_{q}^\times$, define $\ind(x) \in \{ 0, 1, \cdots, q - 2
\}$ such that
\[
x = g^{\ind x}.
\]
Then by definition of $\zeta_{\ell f}$, 
\begin{equation}
\label{Equation:IndZetaEllf}
\ind \zeta_{\ell f} = \frac{q - 1}{\ell f}.
\end{equation}
\end{definition}

\begin{lemma}
\label{Lemma:IndRange}
$\displaystyle \{ \ind x : x \in \mathbf{F}_{q} \setminus \{ 0, 1 \} \} = \{ 1,
2, \dots, q - 2 \}$.
\end{lemma}
\begin{proof}
This is immediate by the definition of $\ind$ since $\ind(1) = 0$.
\end{proof}

\begin{lemma}
\label{Lemma:IndyzAdditive}
For $y, z \in \mathbf{F}_{q}^\times$, $\ind(y z) \equiv \ind y + \ind z \pmod{q
- 1}$.
\end{lemma}
\begin{proof}
This follows immediately from the definition of $\ind$.
\end{proof}

\begin{lemma}
\label{Lemma:CyclotomicProduct}
For $r \in [0, \ell - 1]$ and $j \in [1, f - 1]$,
\[
\sum_{\substack{a \in [1, q - 2] \\ a \equiv j \pmod{f} \\ a \equiv r \pmod
{\ell}}}  \ind(1 - g^{a}) \equiv  \ind \left( 1 - \xi_{\ell}^{r} \xi_{f}^{j}
 \right) \pmod{q - 1}.
\]
\end{lemma}
\begin{proof}
Take the equality 
\begin{align*}
\prod_{k = 0}^{\frac{q - 1}{\ell f} - 1}(1 - g^{k \ell f} X) = 1 -
X^{\frac{q - 1}{\ell f}} \quad\quad\text{in }\mathbf{F}_{q}[X]
\end{align*}
and substitute $X = g^{a}$ to obtain
\begin{align*}
\prod_{k = 0}^{\frac{q - 1}{\ell f} - 1}(1 - g^{a + k \ell f}) &= 1 -
g^{a \left(\frac{q - 1}{\ell f}\right)} \\
&= 1 - \zeta_{\ell f}^{a} \\
&= 1 - \xi_\ell^a \xi_f^a  &\text{(by \autoref{Lemma:ZetaXiConnection})} \\
&= 1 - \xi_\ell^r \xi_f^j ,
\end{align*}
so we are done by taking $\ind$ of both sides and using
\autoref{Lemma:IndyzAdditive}.
\end{proof}

\begin{definition}
For integers $a$ and $b$, define
\[
\delta_{a, b} = \begin{cases}
1 &\text{if } a = b \\
0 &\text{otherwise.}
\end{cases}
\]
\end{definition}

\begin{lemma}\hfill
\label{Lemma:LastIndCongruences}
\begin{enumerate}[label=\upshape(\arabic*),
ref=\autoref{Lemma:LastIndCongruences}(\arabic*)]

\item \label{Lemma:LastIndCongruencesEta0j}
For $m \in [1, f - 1]$,
\[
\eta_{0, m} = 1 - \xi_{f}^{m \ell}.
\]

\item \label{Lemma:LastIndCongruencesXif}
We have 
\begin{align}
\ind \xi_{f} &\equiv 0 \pmod{\ell} \label{Equation:IndXif} \\
\ind \xi_{\ell} &\equiv \frac{q - 1}{\ell f} \pmod{\ell}
\label{Equation:IndXiell}.
\end{align}

\item \label{Lemma:LastIndCongruencesEtaXi}
For $i \in [0, \ell - 1]$ and $j \in [1, f - 1]$,
\[
\ind \eta_{i, j} \equiv \sum_{r = 0}^{\ell - 1} \binom{r}{i} \ind \left( 1 -
\xi_\ell^r \xi_{f}^j \right) \pmod{q - 1}.
\]

\item \label{Lemma:LastIndCongruencesEtaEta}
For $i \in [0, \ell - 1]$ and $j \in [1, f - 1]$,
\begin{align*}
&\ind \eta_{i, f - j}  \\
&\equiv -\delta_{i, \ell - 1} \left( \ind(-1) - \left( \frac{q - 1}{\ell f}
\right) \right) - \delta_{i, \ell - 2} \left( \frac{q - 1}{\ell f} \right) +
(-1)^{i}\sum_{k = 0}^{i}\binom{i - 1}{i - k} \ind \eta_{k, j} \\
&\quad\pmod{\ell}.
\end{align*}

\item \label{Lemma:LastIndCongruencesXiEta}
Suppose that $i \in [1, \ell - 1]$, $j \in [1, f - 1]$, and $m \in [1, f - 1]$
are such that $m \ell \equiv j \pmod{f}$. Then
\[
\ind\left( 1 - \xi_\ell^i \xi_f^j  \right) \equiv \ind \eta_{0, m} - \sum_{s =
\ell - 1 - i}^{\ell - 2} \; \sum_{a = 0}^{s} \binom{s}{a} \ind \eta_{\ell - 2 -
a, j} \pmod{\ell}.
\]

\item \label{Lemma:LastIndCongruencesXiXi}
For $i \in [1, \ell - 1]$ and $j \in [1, f - 1]$,
\[
\ind \left( 1 - \xi_{\ell}^{i} \xi_f^{j}  \right) \equiv \ind(-1) + i \left(
\frac{q - 1}{\ell f} \right) + \ind \left( 1 - \xi_{\ell}^{\ell - i} \xi_f^{f -
j}  \right) \pmod {\ell}.
\]
\end{enumerate}
\end{lemma}
\begin{proof}\hfill
\begin{enumerate}[label=\upshape(\arabic*),
ref={the proof of \autoref{Lemma:LastIndCongruences}(\arabic*)}]

\item \label{LemmaProof:LastIndCongruencesEta0j}
Take the equality
\begin{align*}
\prod_{r = 0}^{\ell - 1} (1 - \xi_{\ell}^r X) = 1 - X^{\ell} \quad\quad\text{in
}\mathbf{F}_{q}[X]
\end{align*}
and substitute $X = \xi_{f}^{m}$ to obtain
\begin{align*}
\eta_{0, m} &= \prod_{r = 0}^{\ell - 1} (1 - \xi_{\ell}^r \xi_{f}^{m})\\
&= 1 - (\xi_{f}^{m})^{\ell}. 
\end{align*}

\item \label{LemmaProof:LastIndCongruencesXif}
Since $\xi_{f}$ is an $f$th root of unity and $\mathbf{F}_{q}$ contains a
primitive $\ell f$th root of unity, $\xi_{f} \in \mathbf{F}_{q}^{\times \ell}$;
\eqref{Equation:IndXif} follows. Taking $\ind$ of both sides of
\autoref{Lemma:ZetaXiConnection} and using \autoref{Lemma:IndyzAdditive} yields
$\ind \zeta_{\ell f} \equiv \ind \xi_{\ell} + \ind \xi_{f} \pmod{q - 1}$, so
\eqref{Equation:IndXiell} follows from \eqref{Equation:IndZetaEllf} and
\eqref{Equation:IndXif}.

\item \label{LemmaProof:LastIndCongruencesEtaXi}
Take $\ind$ of both sides of \autoref{Definition:DefineEtaIJ}.

\item \label{LemmaProof:LastIndCongruencesEtaEta}
Modulo $\ell$, we have
\begin{align}
&\ind \eta_{i, f - j} \nonumber\\
&\equiv \sum_{r = 0}^{\ell - 1} \binom{r}{i} \ind \left( 1 - \xi_\ell^r
\xi_{f}^{- j} \right)  \nonumber\\
&\quad\text{(by \autoref{Lemma:LastIndCongruencesEtaXi})}
\nonumber\\
&\equiv \sum_{r = 0}^{\ell - 1} \binom{r}{i} \left( \ind(-1) + r \ind \xi_{\ell}
- j \ind \xi_{f} + \ind \left( 1 - \xi_\ell^{-r} \xi_{f}^j \right) \right)
\nonumber\\
&\quad\text{(by \autoref{Lemma:IndyzAdditive})} \nonumber\\
&\equiv \ind(-1) \left(\sum_{r = 0}^{\ell - 1} \binom{r}{i} \right) + \left(
\frac{q - 1}{\ell f} \right) \left( \sum_{r = 0}^{\ell - 1} r \binom{r}{i}
\right) + \sum_{r = 0}^{\ell - 1}\ind \left( 1 - \xi_\ell^{-r} \xi_{f}^j \right)
\nonumber\\
&\quad \text{(by \autoref{Lemma:LastIndCongruencesXif})} \nonumber\\
&\equiv \ind(-1) \binom{\ell}{i + 1} + \left(
\frac{q - 1}{\ell f} \right) \left( (\ell - 1)
\binom{\ell}{i + 1} - \binom{\ell}{i + 2} \right) \nonumber\\
&\quad+ \sum_{r = 0}^{\ell - 1}
\binom{r}{i} \ind \left( 1 - \xi_\ell^{-r} \xi_{f}^j \right) \nonumber\\
&\quad \text{(by \autoref{Lemma:BinomialCongruencesHockeystick} and
\autoref{Lemma:BinomialCongruencesPascalModifiedHockeystick})} \nonumber\\
&\equiv \delta_{i, \ell - 1} \left( \ind(-1) - \left( \frac{q - 1}{\ell f}
\right) \right) - \delta_{i, \ell - 2} \left( \frac{q - 1}{\ell f} \right) +
\sum_{r = 0}^{\ell - 1} \binom{r}{i} \ind \left( 1 - \xi_\ell^{-r} \xi_{f}^j
\right),
\label{Equation:ExpandEtaifmj}
\end{align}
since $\binom{\ell}{k}$ is divisible by $\ell$ except when $k \in \{ 0, \ell
\}$, in which case it equals 1 (and we assume that $i \in [0, \ell - 1]$).
Change variables in the last sum to $s \in [0, \ell - 1]$ such that $s \equiv -r
\pmod{\ell}$ (the values $\binom{r}{i}$ and $\xi_{\ell}^{r}$ only depend on $r
\pmod {\ell}$ by \autoref{Lemma:BinomialCongruencesReduceTop} and by definition
of $\xi_{\ell}$).  This yields
\begin{align}
\sum_{r = 0}^{\ell - 1} \binom{r}{i} \ind \left( 1 - \xi_\ell^{-r} \xi_{f}^j
\right) &= \sum_{s = 0}^{\ell - 1} \binom{-s}{i} \ind \left( 1 - \xi_\ell^{s}
\xi_{f}^j \right) \nonumber\\
&\equiv (-1)^{i} \sum_{s = 0}^{\ell - 1}
\sum_{k = 0}^{i} \binom{i - 1}{i - k} \binom{s}{k} \ind \left( 1 - \xi_\ell^{s}
\xi_{f}^j \right) \nonumber\\
&\quad\text{(by \autoref{Lemma:BinomialCongruencesNegative})} \nonumber\\
&\equiv (-1)^{i}\sum_{k = 0}^{i}\binom{i - 1}{i - k} \ind \eta_{k, j}
\label{Equation:FlipSignEta}
\end{align}
by \autoref{Lemma:LastIndCongruencesEtaXi}. We finish by combining
\eqref{Equation:ExpandEtaifmj} and \eqref{Equation:FlipSignEta}.

\item \label{LemmaProof:LastIndCongruencesXiEta}
Modulo $\ell$, we have
\begin{align*}
&\sum_{s = \ell - 1 - i}^{\ell - 2} \; \sum_{a = 0}^{s} \binom{s}{a} \ind
\eta_{\ell - 2 - a, j} \\
&\equiv \sum_{s = \ell - 1 - i}^{\ell - 2} \;  \sum_{r = 0}^{\ell - 1} \;
\sum_{a = 0}^{s} \binom{s}{a} \binom{r}{\ell - 2 - a} \ind\left( 1 -
\xi_{\ell}^r \xi_{f}^{j} \right) &&\text{(by
\autoref{Lemma:LastIndCongruencesEtaXi})} \\
&= \sum_{s = \ell - 1 - i}^{\ell - 2} \;  \sum_{r = 0}^{\ell - 1} \binom{r +
s}{\ell - 2} \ind\left( 1 - \xi_{\ell}^r \xi_{f}^{j} \right) &&\text{(by
\autoref{Lemma:BinomialCongruencesVandermonde})} \\
&\equiv \sum_{s = \ell - 1 - i}^{\ell - 2} \left( \ind\left( 1 -
\xi_{\ell}^{\ell - 2 - s} \xi_{f}^{j} \right) - \ind\left( 1 - \xi_{\ell}^{\ell
- 1 - s} \xi_{f}^{j} \right) \right) &&\text{(by
\autoref{Lemma:BinomialCongruencesReduceTop})}  \\
&= \ind\left( 1 - \xi_{f}^{j} \right) - \ind\left( 1 - \xi_{\ell}^{i}
\xi_{f}^{j} \right) &&\text{(telescoping sum)}\\
&= \ind \eta_{0, m} - \ind\left( 1 - \xi_{\ell}^{i} \xi_{f}^{j} \right)
&&\text{(by \autoref{Lemma:LastIndCongruencesEta0j}).}
\end{align*}

\item \label{LemmaProof:LastIndCongruencesXiXi}
Taking $\ind$ of both sides of $1 - \xi_{\ell}^{i} \xi_f^{j} = -
\xi_{\ell}^{i} \xi_f^{j} \left( 1 - \xi_{\ell}^{\ell - i} \xi_f^{f - j} \right)$
and using \autoref{Lemma:IndyzAdditive} gives 
\begin{align*}
\ind\left( 1 - \xi_{\ell}^{i} \xi_f^{j} \right) &\equiv \ind(-1) + i \ind
\xi_{\ell} + j \ind \xi_{f} + \ind \left( 1 - \xi_{\ell}^{\ell - i} \xi_f^{f -
j} \right) \pmod{\ell} \\
&\equiv \ind(-1) + i \left( \frac{q - 1}{\ell f} \right) + \ind \left( 1 -
\xi_{\ell}^{\ell - i} \xi_f^{f - j} \right) \pmod{\ell}
\end{align*}
by \eqref{Equation:IndXiell} and \eqref{Equation:IndXif}. \qedhere
\end{enumerate}
\end{proof}

\section{Some rings}
\begin{definition}
Define $Q \colonequals \mathbf{Z}[t] / (t^f - 1)$. Define ring homomorphisms
$\alpha \colon Q \to \mathcal{O}_{M}$ and $\beta \colon Q \to \mathbf{Z}$ by
$\alpha(t) = \zeta_f$ and $\beta(t) = 1$. Define
\begin{align*}
&R &&\colonequals Q / \ell Q = \mathbf{Z}[t]/(\ell, t^f - 1) \\
&R' &&\colonequals \text{the subring }\mathbf{Z} / \ell \mathbf{Z} \text{ of } R
\\
&\omega \colon R \to \mathcal{O}_{M} / \ell \mathcal{O}_{M} &&\colonequals
\text{the ring homomorphism induced by }\alpha\text{; i.e., }\omega(t) =
[\zeta_f] \\
&\tau \colon R \to \mathbf{Z} / \ell \mathbf{Z} &&\colonequals
\text{the ring homomorphism induced by }\beta\text{; i.e., }\tau(t) = 1.
\end{align*}
Each $r \in R$ has a unique representation $r = a_0 + a_1 t + \dots + a_{f - 1}
t^{f - 1}$ for $a_0, a_1, \dots, a_{f - 1} \in \mathbf{Z} / \ell \mathbf{Z}$, so
for $j \in [0, f - 1]$, define
\[
[t^{j}] \left( r \right) \colonequals a_{j}
\]
to be the $j$th coefficient of $r$.
\end{definition}

\begin{lemma}
\label{ChineseRemainderTheorem}
The product homomorphism 
\[
(\omega, \tau) : R \to (\mathcal{O}_{M} / \ell \mathcal{O}_{M}) \times
(\mathbf{Z} / \ell \mathbf{Z})
\]
is an isomorphism.
\end{lemma}
\begin{proof}
The ideals $I_1$, $I_2$ of $R$ defined by
\begin{align*}
I_1 &\colonequals (t^{f - 1} + t^{f - 2} + \dots + 1) \\
I_2 &\colonequals (t - 1)
\end{align*}
are pairwise coprime because for
\begin{align*}
&i_1 \colonequals t^{f - 1} + t^{f - 2} + \dots + 1 &&\in I_1 \\
&i_2 \colonequals (t^{f - 1} - 1) + (t^{f - 2} - 1) + \dots + (t - 1) &&\in
I_{2},
\end{align*}
the difference $i_1 - i_2 = f$ is a unit of $R$, so by the Chinese remainder
theorem, the natural map
\[
R \to (R / I_{1}) \times (R / I_{2})
\]
is an isomorphism.  Since $\omega$ is the composite map $\omega \colon R \to R /
I_{1} \simeq \mathcal{O}_{M} / \ell \mathcal{O}_{M}$ and $\tau$ is the composite
map $\tau \colon R \to R / I_{2} \simeq \mathbf{Z} / \ell \mathbf{Z}$, we are
done.
\end{proof}

\begin{lemma}
\label{Lemma:GeneralPropertyInsideR}
For $r \in \ker \tau$, the following are equivalent.
\begin{enumerate}[label=\upshape(\arabic*),
ref=\autoref{Lemma:GeneralPropertyInsideR}(\arabic*)]

\item \label{Lemma:GeneralPropertyInsideRomegar0}
$\omega(r) = 0$;

\item \label{Lemma:GeneralPropertyInsideRr0}
$r = 0$;

\item \label{Lemma:GeneralPropertyInsideRrR'}
$r \in R'$.

\end{enumerate}

\end{lemma}

\begin{proof}
The restriction $\tau|_{R'} : R' \to \mathbf{Z} / \ell \mathbf{Z}$ is an
isomorphism, so $r \in R' \cap \ker \tau$ if and only if $r = 0$, giving
\autoref{Lemma:GeneralPropertyInsideRrR'} $\Longleftrightarrow$
\autoref{Lemma:GeneralPropertyInsideRr0}. By \autoref{ChineseRemainderTheorem},
$r = 0$ if and only if $\tau(r) = 0$ and $\omega(r) = 0$, giving
\autoref{Lemma:GeneralPropertyInsideRomegar0} $\Longleftrightarrow$
\autoref{Lemma:GeneralPropertyInsideRr0}.
\end{proof}

\begin{definition}
For nonnegative integers $u$ and $v$, define
\begin{align*}
&S(u, v) &&\colonequals \sum_{x \in \mathbf{F}_{q} \setminus \{ 0, 1 \}}
\binom{\ind x}{u} \binom{\ind (1 - x)}{v} t^{\ind x} &&\in R \\
&T(u, v) &&\colonequals \tau(S(u, v)) &&\in \mathbf{Z} / \ell \mathbf{Z} \\
&W(u, v) &&\colonequals \omega(S(u, v)) &&\in \mathcal{O}_{M} / \ell
\mathcal{O}_{M}.
\end{align*}
\end{definition}

\begin{lemma}
\label{T0iDivisibleByEll}
For $i \in [0, \ell - 1]$,
\[
T(0, i) = \begin{cases}
-1 & \text{if } i = 0 \\
0 &\text{if } i \in [1, \ell - 2] \\
\frac{q - 1}{\ell} &\text{if } i = \ell - 1.
\end{cases}
\]
\end{lemma}
\begin{proof}
We have
\begin{align*}
T(0, i) &= \tau(S(0, i)) \nonumber \\
&= \tau \left( \sum_{x \in \mathbf{F}_{q} \setminus \{ 0, 1 \}} \binom{\ind (1 -
x)}{i} t^{\ind x} \right) \nonumber \\
&= \sum_{x \in \mathbf{F}_{q} \setminus \{ 0, 1 \}} \binom{\ind (1 - x)}{i}
\nonumber\\
&= \sum_{k = 1}^{q - 2} \binom{k}{i} &&\text{(by \autoref{Lemma:IndRange})} \\
&= \binom{q - 1}{i + 1} - \binom{0}{i} &&\text{(by
\autoref{Lemma:BinomialCongruencesHockeystick}),}
\end{align*}
and the rest follows from \autoref{Lemma:BinomialCongruences0} and
\autoref{Lemma:BinomialCongruencesReduce}.
\end{proof}

\begin{lemma} \hfill
\label{Lemma:W0i0IffS0iR'}

\begin{enumerate}[label=\upshape{(\Alph*)},
ref={\autoref{Lemma:W0i0IffS0iR'}(\Alph*)}]

\item 
\label{Lemma:W0i0IffS0iR'UpToEllM2}
For $i \in [1, \ell - 2]$, the following are equivalent:

\begin{enumerate}[label=\upshape{(\arabic*)},
ref={\theenumi(\arabic*)}]

\item 
\label{Lemma:W0i0IffS0iR'UpToEllM2InR'}
$S(0, i) \in R'$;

\item 
\label{Lemma:W0i0IffS0iR'UpToEllM2W0i0}
$W(0, i) = 0$.

\end{enumerate}

\item
\label{Lemma:W0i0IffS0iR'EllM1}
The following are equivalent:

\begin{enumerate}[label=\upshape{(\arabic*)},
ref={\theenumi(\arabic*)}]

\item 
\label{Lemma:W0i0IffS0iR'EllM1InR'}
$S(0, \ell - 1) - \frac{q - 1}{\ell f} (1 + t + t^2 + \cdots + t^{f - 1}) \in
R'$;

\item
\label{Lemma:W0i0IffS0iR'EllM1W0i0}
$W(0, \ell - 1) = 0$.

\end{enumerate}
\end{enumerate}
\end{lemma}
\begin{proof}
\autoref{T0iDivisibleByEll} implies that 
\begin{align*}
S(0, i) &\in \ker \tau \\
S(0, \ell - 1) - \frac{q - 1}{\ell f} (1 + t + t^2 + \cdots + t^{f - 1}) &\in
\ker \tau,
\end{align*}
so we are done by applying \autoref{Lemma:GeneralPropertyInsideRomegar0}
$\Longleftrightarrow$ \autoref{Lemma:GeneralPropertyInsideRrR'} to $r = S(0, i)$
and to $r = S(0, \ell - 1) - \frac{q - 1}{\ell f} (1 + t + t^2 + \cdots + t^{f -
1})$.
\end{proof}

\section{\texorpdfstring{$\ell$}{ell}-adic valuation of Jacobi sums}
\begin{definition}
For integers $a, b \not\equiv 0 \pmod{\ell f}$, define
\[
J(a, b) \colonequals \sum_{x \in \mathbf{F}_{q} \setminus \{ 0, 1 \}} \zeta_{f
\ell}^{a \ind(x) + b \ind (1 - x)}.
\]
\end{definition}
\begin{lemma} \hfill
\label{Lemma:JellfValuationS0i}

\begin{enumerate}[label=\upshape{(\Alph*)},
ref={\autoref{Lemma:JellfValuationS0i}(\Alph*)}]

\item 
\label{Lemma:JellfValuationS0iUpToEllM1}
For $k \in [1, \ell - 1]$, the following are equivalent:

\begin{enumerate}[label=\upshape{(\arabic*)},
ref={\theenumi(\arabic*)}]

\item 
\label{Lemma:JellfValuationS0iUpToEllM1Congruence}

$J(\ell, f) + 1 \in \pi_{\ell}^{k} \mathcal{O}_{L}$;

\item
\label{Lemma:JellfValuationS0iUpToEllM1InR'}
$S(0, 1),\; S(0, 2),\; \dots,\; S(0, k - 1) \in R'$.

\end{enumerate}

\item 
\label{Lemma:JellfValuationS0iEll}
The following are equivalent:

\begin{enumerate}[label=\upshape{(\arabic*)},
ref={\theenumi(\arabic*)}]

\item 
\label{Lemma:JellfValuationS0iEllCongruence}
$J(\ell, f) + 1 \in \pi_{\ell}^{\ell} \mathcal{O}_{L}$;

\item
\label{Lemma:JellfValuationS0iEllInR'}
$S(0, 1),\; S(0, 2),\; \dots,\; S(0, \ell - 2),\; S(0, \ell - 1) - \frac{q -
1}{\ell f} \left( 1 + t + \cdots + t^{f - 1} \right) \in R'$.

\end{enumerate}
\end{enumerate}
\end{lemma}
\begin{proof}
By definition,
\begin{align*}
&J(\ell, f) \\
&= \sum_{x \in \mathbf{F}_{q} \setminus \{ 0, 1 \}} \zeta_{f
\ell}^{\ell \ind(x) + f \ind(1 - x)} \\
&= \sum_{x \in \mathbf{F}_{q} \setminus \{ 0, 1 \}} \zeta_{f}^{\ind (x)}
\zeta_{\ell}^{\ind(1 - x)}\\ 
&= \sum_{x \in \mathbf{F}_{q} \setminus \{ 0, 1 \}} \zeta_{f}^{\ind (x)}
(1 + \pi_{\ell})^{\ind(1 - x)}\\
&= \sum_{x \in \mathbf{F}_{q} \setminus \{ 0, 1 \}} \zeta_{f}^{\ind (x)} \sum_{i
= 0}^{\ind(1 - x)} \binom{\ind(1 - x)}{i}  \pi_{\ell}^{i}   \\
&= \sum_{x \in \mathbf{F}_{q} \setminus \{ 0, 1 \}} \zeta_{f}^{\ind (x)}
\sum_{i = 0}^{q - 1}  \binom{\ind(1 - x)}{i}  \pi_{\ell}^{i}  &\text{(since
}\ind(1 - x) < q - 1\text{)} \\
&= \sum_{i = 0}^{q - 1} \pi_{\ell}^i \left(\sum_{x \in \mathbf{F}_{q} \setminus
\{ 0, 1 \}} \binom{\ind(1 - x)}{i}\zeta_{f}^{\ind(x)}\right) \\
&\in \left(\sum_{i = 0}^{\ell - 1} \pi_{\ell}^i \left(\sum_{x \in \mathbf{F}_{q}
\setminus \{ 0, 1 \}} \binom{\ind(1 - x)}{i}\zeta_{f}^{\ind(x)}\right)\right) +
\pi_{\ell}^{\ell} \mathcal{O}_{L}
\end{align*}
By \autoref{Lemma:IndRange}, the $i = 0$ term contributes $\zeta_{f}^1 + \dots
+ \zeta_{f}^{q - 2} = (\zeta_{f}^{q - 1} - \zeta_{f}) / (\zeta_{f} - 1) = -1$
since $q \equiv 1 \pmod{\ell f}$, so
\begin{equation}
\label{Equation:ExpandJellf}
J(\ell, f) \in \left(-1 + \sum_{i = 1}^{\ell - 1} \pi_{\ell}^i \left(\sum_{x \in
\mathbf{F}_{q} \setminus \{ 0, 1 \}} \binom{\ind(1 -
x)}{i}\zeta_{f}^{\ind(x)}\right)\right) + \pi_{\ell}^{\ell} \mathcal{O}_{L}
\end{equation}
Since $v_{\ell}(\pi_{\ell}) = \frac{1}{\ell - 1}$, the term $\left(\sum_{x \in
\mathbf{F}_{q} \setminus \{ 0, 1 \}} \binom{\ind(1 -
x)}{i}\zeta_{f}^{\ind(x)}\right)$ lies in $\mathcal{O}_M$, and $M$ is unramified
at $\ell$, the $i$th term in the sum on the right hand side of
\eqref{Equation:ExpandJellf} has $\ell$-adic valuation $\frac{i}{\ell - 1} \pmod
1$. In particular, all the valuations are distinct, so 
\[
J(\ell, f) + 1 \in \pi_{\ell}^{k} \mathcal{O}_{L}
\]
if and only if
\[
\sum_{x \in \mathbf{F}_{q} \setminus \{ 0, 1 \}} \binom{\ind(1 -
x)}{i}\zeta_{f}^{\ind(x)} \in \ell \mathcal{O}_{M} \quad\text{ for }i \in [1, k
- 1],
\]
which is the same as
\[
W(0, 1),\; W(0, 2),\; \cdots,\; W(0, k - 1) = 0,
\]
so we are done by \autoref{Lemma:W0i0IffS0iR'}.
\end{proof}

\section{The connection between \texorpdfstring{$S(i, 1)$}{S(i, 1)} and
cyclotomic units}

Recall that $g$ is a generator for $\mathbf{F}_{q}^{\times}$. We abuse notation
and define $\zeta_{\ell f} \colonequals g^{\frac{q - 1}{\ell f}} \in
\mathbf{F}_{q}^\times$. Using $\zeta_{\ell f}$, define $\zeta_f, \zeta_{\ell},
\xi_f, \xi_{\ell}$ as before.
\begin{lemma}
\label{Si1AndCyclotomicUnitsComputation}
For $i \in [0, \ell - 1]$ and $j \in [1, f - 1]$, 
\[
[t^j] S(i, 1) \equiv \ind \eta_{i, j} \pmod{\ell}.
\]
\end{lemma}
\begin{proof}
By definition of $S(i, 1)$,
\begin{align*}
[t^{j}] S(i, 1) &= \sum_{\substack{a \in [1, q - 2] \\ a \equiv j \pmod{f}}}
\binom{a}{i} \ind(1 - g^{a}) \\
&= \sum_{r = 0}^{\ell - 1} \sum_{\substack{a \in [1, q - 2] \\ a \equiv j \pmod
{f} \\ a \equiv r \pmod{\ell}}} \binom{a}{i} \ind(1 - g^{a}) \\
&\equiv \sum_{r = 0}^{\ell - 1} \binom{r}{i} \sum_{\substack{a \in [1, q - 2] \\
a \equiv j \pmod{f} \\ a \equiv r \pmod{\ell}}}  \ind(1 - g^{a}) \pmod{\ell}
&&\text{(by \autoref{Lemma:BinomialCongruencesReduceTop})} \\
&\equiv \sum_{r = 0}^{\ell - 1} \binom{r}{i} \ind \left( 1 - \xi_{\ell}^{r}
\xi_{f}^{j} \right) \pmod{\ell} &&\text{(by \autoref{Lemma:CyclotomicProduct})}
\\
&\equiv \ind \eta_{i, j} \pmod{\ell} &&\text{(by
\autoref{Lemma:LastIndCongruencesEtaXi}).} \qedhere
\end{align*}
\end{proof}
\begin{lemma} \hfill
\label{Lemma:Si1AndCyclotomicUnits}
\begin{enumerate}[label=\upshape{(\Alph*)},
ref={\autoref{Lemma:Si1AndCyclotomicUnits}(\Alph*)}]

\item 
\label{Lemma:Si1AndCyclotomicUnitsUpToEllM3}

For $i \in [0, \ell - 3]$, the following are equivalent:

\begin{enumerate}[label=\upshape(\arabic*),
ref={\theenumi(\arabic*)}]

\item 
\label{Lemma:Si1AndCyclotomicUnitsUpToEllM3InR'}
$S(i, 1) \in R'$;

\item
\label{Lemma:Si1AndCyclotomicUnitsUpToEllM3Congruence}
$\ind \eta_{i, j} \equiv 0 \pmod{\ell}$ for $j \in [1, f - 1]$.

\end{enumerate}

\item 
\label{Lemma:Si1AndCyclotomicUnitEllM2}

The following are equivalent:

\begin{enumerate}[label=\upshape(\arabic*),
ref={\theenumi(\arabic*)}]

\item 
\label{Lemma:Si1AndCyclotomicUnitsEllM2InR'}
$S(\ell - 2, 1) + \frac{q - 1}{\ell f} \left( 1 + t + \cdots + t^{f - 1} \right)
\in R'$;

\item
\label{Lemma:Si1AndCyclotomicUnitsEllM2Congruence}
$\ind(\eta_{\ell - 2, j}) + \frac{q - 1}{\ell f} \equiv 0 \pmod{\ell}$ for $j
\in [1, f - 1]$.

\end{enumerate}
\end{enumerate}
\end{lemma}
\begin{proof}
For any $r \in R$, the condition $r \in R'$ is equivalent to $[t^{j}] r \equiv
0 \pmod{\ell}$ for $j \in [1, f - 1]$. Apply this observation to $r \in \{ S(0,
1), \cdots, S(\ell - 3), S(\ell - 2, 1) + \frac{q - 1}{\ell f} \left( 1 + t +
\cdots + t^{f - 1} \right) \}$ and use
\autoref{Si1AndCyclotomicUnitsComputation} to finish.
\end{proof}

\section{A recursion for \texorpdfstring{$S(u, v)$}{S(u, v)}}

In this section, we will investigate the product of expressions of the form
$S(u, v)$. 

\begin{lemma}
\label{Lemma:RecursionForSuv}
For $i \in [1, \ell - 2]$ and $s \in [1, i]$,
\begin{align*}
&(i - s + 1) S(i - s + 1, s) - (s + 1) S(i - s, s + 1) \\
&\equiv
\left(\sum_{r = 0}^{i - s} S(i - s - r, s) S(r, 1)\right) - \left(\sum_{k =
1}^{s}  T(1, s - k) S(i - s, k)\right) - (i - 2 s) S(i - s, s) \\
&\pmod {R'}.
\end{align*}
\end{lemma}
\begin{proof}
By definition of $S(u, v)$,
\begin{align*}
&\sum_{r = 0}^{i - s} S(i - s - r, s) S(r, 1) \\
&= \sum_{y, z \in \mathbf{F}_{q} \setminus \{ 0, 1 \} } \sum_{r = 0}^{i - s}
\binom{\ind(y)}{i - s - r} \binom{\ind(z)}{r} \binom{\ind(1 - y)}{s} \ind(1 - z)
t^{\ind y} t^{\ind z} \\
&= \sum_{y, z \in \mathbf{F}_{q} \setminus \{ 0, 1 \} } 
\binom{\ind(y) + \ind(z)}{i - s} \binom{\ind(1 - y)}{s} \ind(1 - z)
t^{\ind y + \ind z} \\
&\qquad\text{(by \autoref{Lemma:BinomialCongruencesVandermonde})} \\
&= \sum_{y, z \in \mathbf{F}_{q} \setminus \{ 0, 1 \} } 
\binom{\ind(y z)}{i - s} \binom{\ind(1 - y)}{s} \ind(1 - z)
t^{\ind (y z)} \\
&\qquad\text{(by \autoref{Lemma:IndyzAdditive},
\autoref{Lemma:BinomialCongruencesReduceTop}, and
} t^{q - 1} = 1\text{)} \\
&= \sum_{\substack{x \in \mathbf{F}_{q} \setminus \{ 0 \} \\ y \in
\mathbf{F}_{q} \setminus \{ 0, 1, x \} }} \binom{\ind (1 - y)}{s} \ind\left( 1 -
\frac{x}{y} \right) \binom{\ind x}{i - s} t^{\ind x} \\
&\qquad\text{(by setting }x \colonequals y z\text{)} \\
&= \sum_{\substack{x \in \mathbf{F}_{q} \setminus \{ 0 \} \\ y \in
\mathbf{F}_{q} \setminus \{ 0, 1, x \} }} \binom{\ind (1 - y)}{s} \left( \ind(y
- x) - \ind(y) \right) \binom{\ind x}{i - s} t^{\ind x} \\
&\qquad\text{(by \autoref{Lemma:IndyzAdditive})} \\
&\equiv \sum_{\substack{x \in \mathbf{F}_{q} \setminus \{ 0, 1 \} \\ y \in
\mathbf{F}_{q} \setminus \{ 0, 1, x \} }} \binom{\ind (1 - y)}{s} \left( \ind(y
- x) - \ind(y) \right) \binom{\ind x}{i - s} t^{\ind x} \pmod{R'},
\end{align*}
so if we define
\begin{align*}
A &\colonequals \sum_{\substack{x \in \mathbf{F}_{q} \setminus \{ 0, 1 \} \\ y
\in \mathbf{F}_{q} \setminus \{ 0, 1, x \} }} \binom{\ind (1 - y)}{s} \ind(y -
x) \binom{\ind x}{i - s} t^{\ind x} \\
B &\colonequals \sum_{\substack{x \in \mathbf{F}_{q} \setminus \{ 0, 1 \} \\ y
\in \mathbf{F}_{q} \setminus \{ 0, 1, x \} }} \binom{\ind (1 - y)}{s} \ind(y)
\binom{\ind x}{i - s} t^{\ind x},
\end{align*}
then
\begin{equation}
\label{Equation:AminusB}
\sum_{r = 0}^{i - s} S(i - s - r, s) S(r, 1) \equiv A - B \pmod{R'}.
\end{equation}
We have
\begin{align*}
B &= \sum_{\substack{x \in \mathbf{F}_{q} \setminus \{ 0, 1 \} \\ y \in
\mathbf{F}_{q} \setminus \{ 0, 1 \} }} \binom{\ind (1 - y)}{s} \ind(y)
\binom{\ind x}{i - s} t^{\ind x} \\
&\quad- \sum_{\substack{x \in \mathbf{F}_{q}
\setminus \{ 0, 1 \} \\ y = x }} \binom{\ind (1 - y)}{s} \ind(y) \binom{\ind
x}{i - s} t^{\ind x} \\
&= \left( \sum_{y \in \mathbf{F}_q \setminus \{ 0, 1 \}} \ind(y) \binom{\ind (1
- y)}{s} \right) \left(\sum_{x \in \mathbf{F}_q \setminus \{ 0, 1 \}}
\binom{\ind x}{i - s} t^{\ind x} \right)  \\
&\quad- \sum_{x \in \mathbf{F}_{q} \setminus \{ 0, 1 \}} \binom{\ind (1 - x)}{s}
\ind(x) \binom{\ind x}{i - s} t^{\ind x} \\ &= T(1, s) S(i - s, 0) \\
&\quad- \sum_{x \in \mathbf{F}_{q} \setminus \{ 0, 1 \}}
\binom{\ind (1 - x)}{s} \left( (i - s + 1) \binom{\ind x}{i - s + 1} + (i - s)
\binom{\ind x}{i - s}  \right) t^{\ind x} \\
&\qquad \text{(by definition of }T(1, s),\, S(i - s, 0)\text{, and
\autoref{Lemma:BinomialCongruencesPascalModified})} \\
&= T(1, s) S(i - s, 0) - (i - s + 1) S(i - s + 1, s) - (i - s) S(i - s, s) \\
&\qquad \left(\text{by definition of }S(i - s + 1, s)\text{ and } S(i - s,
s)\right).
\end{align*}
Since $s \ge 1$, the summand in $A$ vanishes when $y = 0$, so we can put it back
in to get
\begin{align*}
A &= \sum_{\substack{x \in \mathbf{F}_{q} \setminus \{ 0, 1 \} \\ y \in
\mathbf{F}_{q} \setminus \{ 1, x \} }} \binom{\ind (1 - y)}{s} \ind(y - x)
\binom{\ind x}{i - s} t^{\ind x} \\
&= \sum_{\substack{x \in \mathbf{F}_{q} \setminus \{ 0, 1 \} \\ w \in
\mathbf{F}_{q} \setminus \{ 0, 1 \} }} \binom{\ind ( (1 - x)(1 - w) )}{s} \ind(
(1 - x) w) \binom{\ind x}{i - s} t^{\ind x}\\
&\qquad\text{(by setting } w \colonequals (x - y)/(x - 1) \text{)} \\
&= \sum_{k = 0}^{s} \sum_{\substack{x \in \mathbf{F}_{q} \setminus \{ 0, 1 \} \\
w \in \mathbf{F}_{q} \setminus \{ 0, 1 \} }} \binom{\ind(1 - x)}{k}
\binom{\ind(1 - w)}{s - k} \ind\left( 1 - x \right) \binom{\ind x}{i - s}
t^{\ind x} \\
&\quad+ \sum_{k = 0}^{s} \sum_{\substack{x \in \mathbf{F}_{q} \setminus \{ 0, 1
\} \\ w \in \mathbf{F}_{q} \setminus \{ 0, 1 \} }} \binom{\ind(1 - x)}{k}
\binom{\ind(1 - w)}{s - k} \ind\left( w \right) \binom{\ind x}{i - s}
t^{\ind x} \\
&\qquad\text{(by \autoref{Lemma:IndyzAdditive},
\autoref{Lemma:BinomialCongruencesReduceTop}, and
\autoref{Lemma:BinomialCongruencesVandermonde})} \\
&= \sum_{k = 0}^{s} \sum_{\substack{x \in \mathbf{F}_{q} \setminus \{ 0, 1 \} \\
w \in \mathbf{F}_{q} \setminus \{ 0, 1 \} }} \left( (k + 1) \binom{\ind(1 -
x)}{k + 1} + k \binom{\ind(1 - x)}{k} \right) \binom{\ind(1 - w)}{s - k}
\binom{\ind x}{i - s} t^{\ind x} \\
&\quad+ \sum_{k = 0}^{s} \sum_{\substack{x \in \mathbf{F}_{q} \setminus \{ 0, 1
\} \\ w \in \mathbf{F}_{q} \setminus \{ 0, 1 \} }} \binom{\ind(1 - x)}{k}
\binom{\ind(1 - w)}{s - k} \ind\left( w \right) \binom{\ind x}{i - s}
t^{\ind x} \\
&\qquad\text{(by \autoref{Lemma:BinomialCongruencesPascalModified})} \\
&= \sum_{k = 0}^{s} \left[\left( \sum_{w \in \mathbf{F}_{q} \setminus \{ 0, 1 \}}
\binom{\ind(1 - w)}{s - k} \right)\right. \\
&\quad\left.\times \left( \sum_{x \in \mathbf{F}_{q} \setminus \{ 0, 1 \}}
\binom{\ind x}{i - s} \left( (k + 1) \binom{\ind(1 - x)}{k + 1} + k
\binom{\ind(1 - x)}{k} \right)  t^{\ind x} \right)\right] \\
&\quad+ \sum_{k = 0}^{s} \left(\sum_{w \in \mathbf{F}_{q} \setminus \{ 0, 1 \}}
\ind\left( w \right) \binom{\ind(1 - w)}{s - k}  \right) \left(\sum_{x \in
\mathbf{F}_{q} \setminus \{ 0, 1 \}} \binom{\ind x}{i - s} \binom{\ind(1 -
x)}{k} t^{\ind x}\right) \\
&= \left(\sum_{k = 0}^{s} T(0, s - k) \left( (k + 1) S(i - s, k + 1) + k S(i -
s, k) \right)\right) + \sum_{k = 0}^{s} T(1, s - k) S(i - s, k) \\
&\quad\left(\text{by definition of }S(u, v) \text{ and } T(u, v)\right) \\
&= -(s + 1) S(i - s, s + 1) - s S(i - s, s) + \sum_{k = 0}^{s} T(1, s - k) S(i -
s, k) \\
&\quad\left(\text{by \autoref{T0iDivisibleByEll}}\right) \\
&= -(s + 1) S(i - s, s + 1) - s S(i - s, s) + T(1, s) S(i - s, 0) + \sum_{k =
1}^{s} T(1, s - k) S(i - s, k), 
\end{align*}
and we finish by substituting these expressions for $A$ and $B$ into
\eqref{Equation:AminusB}.
\end{proof}

\begin{corollary}
\label{Corollary:RecursionEll}
Suppose that $i \in [1, \ell - 2]$. Assume that $S(u, v) \in R'$ holds whenever
$u + v \le i$ and $v \ge 1$. Then for all $s \in [1, i]$,
\[
(i - s + 1) S(i - s + 1, s) \equiv (s + 1) S(i - s, s + 1) \pmod{R'}.
\]
\end{corollary}
\begin{proof}
The assumptions imply that all the terms on the right hand side of
\autoref{Lemma:RecursionForSuv} lie in $R'$, so \autoref{Lemma:RecursionForSuv}
implies the corollary.
\end{proof}

\begin{corollary}
\label{Corollary:RecursionEllM1}
Suppose that $i \in [1, \ell - 2]$.  Assume that $S(u, v) \in R'$ holds whenever
$u + v \le i$ and $v \ge 1$. Then if one of 
\[
S(i, 1),\; S(i - 1, 2),\; \ldots,\; S(0, i + 1)
\]
is in $R'$, then they must all be in $R'$.
\end{corollary}
\begin{proof}
For $s \in [1, i]$, \autoref{Corollary:RecursionEll} implies
\[
(i - s + 1) S(i - s + 1, s) \equiv (s + 1) S(i - s, s + 1) \pmod {R'},
\]
so since $i - s + 1$ and $s + 1$ are invertible modulo $\ell$ (they lie in $[1,
\ell - 1]$), 
\[
S(i - s + 1, s) \in R'\text{ if and only if }S(i - s, s + 1) \in R'.
\]
Since this holds for all $s \in [1, i]$, we are done. 
\end{proof}
\section{Proof of main theorem}
\label{MainTheoremSection}

Now we combine all of our results from the previous sections in the following
lemma.
\begin{lemma}
\label{Lemma:UpToEllMinus1}
For $k \in [1, \ell - 1]$, the following are equivalent:
\begin{enumerate}[label=\upshape(\arabic*),
ref=\autoref{Lemma:UpToEllMinus1}(\arabic*)]

\item \label{Lemma:UpToEllMinus1Su1}
$S(0, 1), S(1, 1), \cdots, S(k- 2, 1)$ lie in $R'$;

\item \label{Lemma:UpToEllMinus1Suv}
$S(u, v)$ lies in $R'$ for $u \ge 0$ and $v \ge 1$ satisfying $u + v
\le k - 1$;

\item \label{Lemma:UpToEllMinus1S0v}
$S(0, 1), S(0, 2), \cdots, S(0, k - 1)$ lie in $R'$;

\item \label{Lemma:UpToEllMinus1Congruence}
$J(\ell, f) + 1 \in \pi_{\ell}^{k} \mathcal{O}_{L}$.

\end{enumerate}
\end{lemma}
\begin{proof}
\autoref{Corollary:RecursionEllM1} implies that conditions
\cref{Lemma:UpToEllMinus1Su1,Lemma:UpToEllMinus1Suv,Lemma:UpToEllMinus1S0v} are
equivalent. By \autoref{Lemma:JellfValuationS0iUpToEllM1}, conditions
\cref{Lemma:UpToEllMinus1S0v,Lemma:UpToEllMinus1Congruence} are equivalent.
\end{proof}

\CongruenceUpToEllMinusOne
\begin{proof}
Combine \autoref{Lemma:Si1AndCyclotomicUnitsUpToEllM3} and
\autoref{Lemma:UpToEllMinus1Su1} $\Longleftrightarrow$
\autoref{Lemma:UpToEllMinus1Congruence} for
\autoref{Theorem:CongruenceUpToEllMinus1Congruence} $\Longleftrightarrow$
\autoref{Theorem:CongruenceUpToEllMinus1Units}.

\autoref{Lemma:LastIndCongruencesEtaEta} implies that for $i \in [0, \ell - 3]$
and $j \in [1, f / 2]$, $\ind \eta_{i, f - j}$ is a linear combination of
$\ind_{0, j}, \dots, \ind_{i, j}$ modulo $\ell$, and this implies
\autoref{Theorem:CongruenceUpToEllMinus1Units} $\Longleftrightarrow$
\autoref{Theorem:CongruenceUpToEllMinus1UnitsHalf}.
\end{proof}

\section{The case \texorpdfstring{$k = \ell$}{k = ell}}
\label{Section:kEqualsEll}

\begin{lemma}
\label{Lemma:ReductionToDivisibilityInRpIfkEqualsEll2}
The following are equivalent.

\begin{enumerate}[label=\upshape{(\arabic*)},
ref={\autoref{Lemma:ReductionToDivisibilityInRpIfkEqualsEll2}(\arabic*)}]

\item 
\label{Lemma:ReductionToDivisibilityInRpIfkEqualsEll2Congruence}

$J(\ell, f) + 1 \in \pi_{\ell}^{\ell}\mathcal{O}_{L}$;

\item
\label{Lemma:ReductionToDivisibilityInRpIfkEqualsEll2R'}
$S(0, 1),\; S(1, 1),\; \cdots,\; S(\ell - 3, 1),\; S(\ell - 2, 1) +
\displaystyle\frac{q - 1}{\ell f} (1 + t + t^2 + \cdots + t^{f - 1}) \in R'$.

\end{enumerate}
\end{lemma}
\begin{proof}

By \autoref{Lemma:JellfValuationS0iEll},

\begin{itemize}
\item $J(\ell, f) + 1 \in \pi_{\ell}^{\ell}\mathcal{O}_{L}$
\end{itemize}
is equivalent to
\begin{itemize}
\item $S(0, 1)$, $S(0, 2)$, \ldots, $S(0, \ell - 2)$ lie in $R'$, and
\item $S(0, \ell - 1) - \frac{q - 1}{\ell f} (1 + t + \cdots + t^{f - 1})$ lies
in $R'$,
\end{itemize}
which by \autoref{Lemma:UpToEllMinus1S0v}
$\Longleftrightarrow$ \autoref{Lemma:UpToEllMinus1Suv}, is equivalent to
\begin{itemize}
\item for $u \ge 0$ and $v \ge 1$ satisfying $u + v
\le \ell - 2$, $S(u, v)$ lies in $R'$, and
\item $S(0, \ell - 1) - \frac{q - 1}{\ell f} (1 + t + \cdots + t^{f - 1})$ lies
in $R'$,
\end{itemize}
which by \autoref{Corollary:RecursionEll}, is equivalent to
\begin{itemize}
\item for $u \ge 0$ and $v \ge 1$ satisfying $u + v
\le \ell - 2$, $S(u, v)$ lies in $R'$, 
\item for all $s \in [1, \ell - 2]$, $(\ell - 1 - s) S(\ell - 1 - s, s) \equiv
(s + 1) S(\ell - 2 - s, s + 1) \pmod{R'}$, and
\item $S(0, \ell - 1) - \frac{q - 1}{\ell f} (1 + t + \cdots + t^{f - 1})$ lies
in $R'$,
\end{itemize}
which is equivalent to
\begin{itemize}
\item for $u \ge 0$ and $v \ge 1$ satisfying $u + v
\le \ell - 2$, $S(u, v)$ lies in $R'$, 
\item for all $s \in [1, \ell - 2]$, $(\ell - 1 - s) S(\ell - 1 - s, s) \equiv
(s + 1) S(\ell - 2 - s, s + 1) \pmod{R'}$, and
\item $S(\ell - 2, 1) - (-1)^{\ell - 2} \frac{q - 1}{\ell f} (1 + t + \cdots +
t^{f - 1})$ lies in $R'$,
\end{itemize}
which by \autoref{Corollary:RecursionEll}, is equivalent to
\begin{itemize}
\item for $u \ge 0$ and $v \ge 1$ satisfying $u + v
\le \ell - 2$, $S(u, v)$ lies in $R'$, 
\item $S(\ell - 2, 1) - (-1)^{\ell - 2} \frac{q - 1}{\ell f} (1 + t + \cdots +
t^{f - 1})$ lies in $R'$,
\end{itemize}
which by \autoref{Lemma:UpToEllMinus1Suv} $\Longleftrightarrow$
\autoref{Lemma:UpToEllMinus1Su1}, is equivalent to
\begin{itemize}
\item $S(0, 1)$, $S(1, 1)$, \ldots, $S(\ell - 3, 1)$ lies in $R'$, 
\item $S(\ell - 2, 1) - (-1)^{\ell - 2} \frac{q - 1}{\ell f} (1 + t + \cdots +
t^{f - 1})$ lies in $R'$,
\end{itemize}
and we are done by observing that $(-1)^{\ell - 2} \equiv -1 \pmod{\ell}$.
\qedhere

\end{proof}
\begin{lemma}
\label{Lemma:IndexDivisibilityForKEqualsEll}
The following are equivalent:
\begin{enumerate}[label=\upshape{(\arabic*)},
ref={\autoref{Lemma:IndexDivisibilityForKEqualsEll}(\arabic*)}]

\item 
\label{Lemma:IndexDivisibilityForKEqualsEllJacobiCongruence}
$J(\ell, f) + 1 \in \pi_{\ell}^{\ell} \mathcal{O}_{L}$;

\item
\label{Lemma:IndexDivisibilityForKEqualsEllIndCongruence}
All the following are divisible by $\ell$:
\[
\begin{tabu}{cccc}
\ind(\eta_{0, 1}) &\ind(\eta_{0, 2}) &\ldots &\ind(\eta_{0, f - 1}) \\
\ind(\eta_{1, 1}) &\ind(\eta_{1, 2}) &\ldots &\ind(\eta_{1, f - 1}) \\
\vdots &\vdots &\ddots &\vdots \\
\ind(\eta_{\ell - 3, 1}) &\ind(\eta_{\ell - 3, 2}) &\ldots &\ind(\eta_{\ell - 3,
f - 1}) \\
\ind(\eta_{\ell - 2, 1}) + \frac{q - 1}{\ell f} &\ind(\eta_{\ell - 2, 2}) +
\frac{q - 1}{\ell f} &\ldots &\ind(\eta_{\ell - 2, f - 1}) + \frac{q - 1}{\ell
f}
\end{tabu}
\]

\end{enumerate}

\end{lemma}
\begin{proof}
Combine \autoref{Lemma:Si1AndCyclotomicUnitEllM2} and
\autoref{Lemma:ReductionToDivisibilityInRpIfkEqualsEll2}. 
\end{proof}

\begin{corollary}
\label{Corollary:SetupCongruencekEqualsEll}
The following are equivalent:

\begin{enumerate}[label=\upshape(\arabic*),
ref=\autoref{Corollary:SetupCongruencekEqualsEll}(\arabic*)]

\item \label{Corollary:SetupCongruencekEqualsEllCongruence}
$J(\ell, f) + 1 \in \pi_{\ell}^{\ell} \mathcal{O}_{L}$;

\item \label{Corollary:SetupCongruencekEqualsEllUnits}
$\frac{q - 1}{\ell f}\equiv 0 \pmod{\ell}$ and $\ind(1 - \xi_{\ell}^{i}
\xi_{f}^{j}  ) \equiv 0 \pmod{\ell}$ for all $i \in [0, \ell - 1]$ and $j \in [1, f -
1]$;

\item \label{Corollary:SetupCongruencekEqualsEllUnitsHalf}
$\frac{q - 1}{\ell f}\equiv 0 \pmod{\ell}$ and $\ind(1 - \xi_{\ell}^{i}
\xi_{f}^{j}  ) \equiv 0 \pmod{\ell}$ for all for all $i \in [0, \ell - 1]$ and
$j \in [1, f / 2]$.

\end{enumerate}

\end{corollary}
\begin{proof} \hfill

\begin{enumerate}[label=\upshape{(\alph*)}] 

\item 
\autoref{Corollary:SetupCongruencekEqualsEllUnits} $\Longrightarrow$
\autoref{Corollary:SetupCongruencekEqualsEllUnitsHalf}

This is obvious.

\item 
\autoref{Corollary:SetupCongruencekEqualsEllUnits} $\Longrightarrow$
\autoref{Corollary:SetupCongruencekEqualsEllCongruence}

This follows from \autoref{Lemma:LastIndCongruencesEtaXi} and
\autoref{Lemma:IndexDivisibilityForKEqualsEllIndCongruence} $\Longrightarrow$
\autoref{Lemma:IndexDivisibilityForKEqualsEllJacobiCongruence}.

\item

\autoref{Corollary:SetupCongruencekEqualsEllUnitsHalf} $\Longrightarrow$
\autoref{Corollary:SetupCongruencekEqualsEllUnits} 

Suppose that $i \in [0, \ell - 1]$ and $j \in [f / 2, f - 1]$. Then
\begin{align}
&\ind \left( 1 - \xi_{\ell}^{i} \xi_f^{j}  \right) \nonumber \\
&\equiv \ind(-1) + i \left( \frac{q - 1}{\ell f} \right) + \ind \left( 1 -
\xi_{\ell}^{\ell - i} \xi_f^{f - j} \right) \pmod {\ell} \quad\quad\text{(by
\autoref{Lemma:LastIndCongruencesXiXi})} \nonumber \\
&\equiv \ind(-1) \quad\quad\left(\text{since } \frac{q - 1}{\ell f} \equiv 0
\pmod{\ell} \text{ and } f - j \in [1, f / 2] \right).
\label{Equation:IndReducedToIndM1}
\end{align}
If $\ell = 2$, then $\ind(-1) = (q - 1) / 2 = (q - 1) / \ell \equiv 0
\pmod{\ell}$ since $q - 1 \equiv 0 \pmod{\ell^{2} f}$ by assumption. If $\ell$
is odd, then $\ind(-1) = (q - 1) / 2 \equiv 0 \pmod{\ell}$ since $q - 1 \equiv 0
\pmod{\ell}$ and $2$ is coprime to $\ell$. In any case, $\ind(-1) \equiv 0
\pmod{\ell}$ so we are done by \eqref{Equation:IndReducedToIndM1}.

\item

\autoref{Corollary:SetupCongruencekEqualsEllCongruence} $\Longrightarrow$
\autoref{Corollary:SetupCongruencekEqualsEllUnits} 

\autoref{Corollary:SetupCongruencekEqualsEllCongruence} is
\autoref{Lemma:IndexDivisibilityForKEqualsEllJacobiCongruence}, so
\autoref{Lemma:IndexDivisibilityForKEqualsEll} implies that
\autoref{Lemma:IndexDivisibilityForKEqualsEllIndCongruence} holds.  Combining
\autoref{Lemma:IndexDivisibilityForKEqualsEllIndCongruence} with
\autoref{Lemma:LastIndCongruencesEtaEta} with $i = \ell - 2$ (and any value of
$j$) yields
\[
-\left( \frac{q - 1}{\ell f} \right) \equiv - \left( \frac{q -
1}{\ell f} \right) - (-1)^{\ell - 2} \left( \frac{q - 1}{\ell f} \right)
\pmod{\ell},
\]
which implies
\[
\frac{q - 1}{\ell f} \equiv 0 \pmod{\ell}.
\]
Combining this with \autoref{Lemma:IndexDivisibilityForKEqualsEllIndCongruence}
implies that $\ind \eta_{k, j} \equiv 0 \pmod{\ell}$ for all $k \in [0, \ell -
2]$ and $j \in [1, f - 1]$, so \autoref{Lemma:LastIndCongruencesXiEta} gives
that $\ind\left( 1 - \xi_\ell^i \xi_f^j  \right) \equiv 0 \pmod{\ell}$ for all
$i \in [1, \ell - 1]$ and $j \in [1, f - 1]$. \qedhere
\end{enumerate}
\end{proof}

\CongruencekEqualsEll
\begin{proof}
This is a restatement of \autoref{Corollary:SetupCongruencekEqualsEll}.
\end{proof}


\end{document}